%% file: main.tex
\documentclass[a4paper,10pt]{article}

\usepackage{amsfonts}
\usepackage{amssymb}
\usepackage{amsmath, amsfonts, amsxtra}
\usepackage{verbatim} %
\usepackage{hyperref}%

\input{environments}

\input{commands}

\numberwithin{equation}{section}%

\title{Boundary regularity via Uhlenbeck-Rivi\`ere decomposition}
\author{Frank M\"uller\footnote{F. M\"uller: Universit\"at Duisburg-Essen, Fachbereich Mathematik, 47048 Duisburg, Germany (mueller@math.tu-cottbus.de)}, Armin Schikorra\footnote{A. Schikorra: RWTH Aachen, Institut f\"ur Mathematik, 52064 Aachen, Germany (schikorra@instmath.rwth-aachen.de)}}
\date{}

\begin{document}
\maketitle
\begin{abstract}
\noindent We prove that weak solutions of systems with skew-symmetric structure, which possess a continuous boundary trace, have to be continuous up to the boundary. This applies, e.g., to the $H$-surface system $\lap u = 2H(u)\partial_{x^1}u\wedge\partial_{x^2} u$ with bounded $H$ and thus extends an earlier result by P.\,Strzelecki and proves the natural counterpart of a conjecture by E.\,Heinz. Methodically, we use estimates below natural exponents of integrability and a recent decomposition result by T.\,Rivi\`ere.\\[1ex]

\end{abstract}

\input{introduction-2}

\input{calculations-1}
\input{appendix-1}

\bibliographystyle{alpha}%
\bibliography{bib}%

\end{document}

%% file: environments.tex
\newtheorem{theorem}{Theorem}[section]%
\newtheorem{lemma}[theorem]{Lemma}%
\newtheorem{proposition}[theorem]{Proposition}%
\newtheorem{corollary}[theorem]{Corollary}%
\newtheorem{definition}[theorem]{Definition}%
\newtheorem{remark}[theorem]{Remark}%

\newenvironment{proof}[1][Proof.]{\begin{trivlist}%
\item[\hskip \labelsep {\bfseries #1}]}{\end{trivlist}}%

\usepackage{array} 
\newenvironment{ma}{\begin{array}{>{\displaystyle}r >{\displaystyle}c >{\displaystyle}l}}{\end{array}}%

%% file: commands.tex
\newcommand{\N}{{\mathbb N}}
\newcommand{\R}{{\mathbb R}}

\newcommand{\lap}{\triangle}
\newcommand{\abs}[1]{{\left \vert #1 \right \vert}}%

\newcommand{\qed}{\hfill \ensuremath{\Box}}%

\newcommand{\disk}{D^2}
\newcommand{\strich}{{'}}

\def\Xint#1{\mathchoice%
   {\XXint\displaystyle\textstyle{#1}}
   {\XXint\textstyle\scriptstyle{#1}}
   {\XXint\scriptstyle\scriptscriptstyle{#1}}
   {\XXint\scriptscriptstyle\scriptscriptstyle{#1}}
   \!\int}%
\def\XXint#1#2#3{{\setbox0=\hbox{$#1{#2#3}{\int}$}%
     \vcenter{\hbox{$#2#3$}}\kern-.5\wd0}}%
\def\mvint{\Xint-}%

\newcommand{\constant}[1][]{C_{#1}}
\newcommand{\eps}[1][]{\varepsilon_{#1}}

%% file: introduction-2.tex
\section{Introduction}
In the present paper we establish regularity up to the boundary for weak solutions of a class of second order equations with continuous boundary trace. Let us start with a typical example: Write $D^2:=\{x=(x^1,x^2)\in\R^2\,:\ |x|<1\}$ for the unit disc in $\R^2$ and consider a weak solution $u\in W^{1,2}(D^2,\R^3)$ of the \emph{$H$-surface system}
\begin{equation}
\label{eq:h-surface}
 -\Delta u=-2H(u)\partial_{x^1}u\wedge\partial_{x^2}u\quad\mbox{in}\ D^2
\end{equation}
with some prescribed function $H\in L^\infty(\R^3)$. Proving a conjecture by E.\,Heinz \cite{Heinz}, T.\,Rivi\`ere \cite{Riviere06} showed that $u$ then has to be continuous in $D^2$; see Rivi\`ere's paper for a list of several earlier attempts in proving Heinz' conjecture. The importance of (\ref{eq:h-surface}) comes from the fact that conformally parametrized solutions of (\ref{eq:h-surface}) form surfaces with prescribed mean curvature $H$ in $\R^3$. 

A natural counterpart of Heinz' conjecture is the following: Assume that $u\in W^{1,2}(D^2,\R^3)$ solves (\ref{eq:h-surface}) and, additionally, that $u|_{\partial D^2}$ is continuous. Can we then prove that $u$ belongs to $C^0(\overline{D^2})$, merely supposing that $H\in L^\infty(\R^3)$? There are several partial answers to this question in the literature, e.g., \cite{BrC84}, \cite{Jakobowsky}, \cite{Strzelecki03}, \cite{Chone}, where additional assumptions on $H$ (constant, Lipschitz, structure conditions)  or $u$ (a priori bounded) were presupposed; see also \cite{Qing}, where boundary regularity for bounded, weakly harmonic maps is proved. We now can settle the posed question completely as a corollary of our main theorem:

\begin{theorem}\label{th:boundreg}
Let $\Omega \in L^2(D^2,so_m\otimes \R^2)$ and $e \in L^s(D^2,\R^m)$, $s>1$, be given. Then, any weak solution $u \in W^{1,2}(D^2,\R^m)$ of
\begin{equation}
\label{eq:sys}
 -\lap u = \Omega \cdot \nabla u + e \mbox{\quad in $D^2$}
 \end{equation}
belongs to $C^{0,\alpha}(D^2,\R^m)$ for some $\alpha > 0$. If the trace $u\big \vert_{\partial D^2}$ is continuous, then we conclude $u \in C^{0,\alpha} (D^2,\R^m) \cap C^0(\overline{D^2},\R^m)$.
\end{theorem}

In Theorem\,\ref{th:boundreg}, $so_m\otimes \R^2$ denotes the space of skew-symmetric $m\times m$-matrices with entries in $\R^2$, $\nabla=(\partial_{x^1},\partial_{x^2})^t$ is the gradient and $\Omega\cdot\nabla u$ stands for the matrix product with entries given by the scalar product of the respective components of $\Omega$ and $\nabla u$. 

\begin{remark}
The main new contribution in Theorem\,\ref{th:boundreg} is the continuity result up to the Dirichlet type boundary. The interior regularity was proved by T.\,Rivi{\`e}re in \cite{Riviere06} (for $e\equiv0$) and our proof is based on Rivi\`ere's decomposition result combined with the Dirichlet growth approach by Rivi\`ere and Struwe in \cite{StruweRiviere} as well as some additional arguments due to P.\,Strzelecki \cite{Strzelecki03}. 
\end{remark}

\begin{remark}
Let us emphasize that one can prove Theorem\,\ref{th:boundreg} also by reflection across $\partial D^2$, whenever there is some $\psi \in W^{2,p}(D^2,\R^m)$, $p > 1$, such that $u = \psi$ on $\partial D^2$. Indeed, the difference function $v:=u-\psi\in W^{1,2}_0(D^2,\R^m)$ then also solves system (\ref{eq:sys}) with a zero order term $\tilde e \in L^{\tilde s}(D^2,\R^m)$ for some $\tilde s > 1$. Odd reflection of $v$ and appropriate reflection of the data $\Omega,\tilde e$ then yields an analogue system on some larger disc $B_{1+\delta}(0)$, $\delta>0$, and the assertion follows by Rivi\`ere's interior regularity result.
\end{remark}

\begin{remark}
In Theorem\,\ref{th:boundreg} the unit disc $D^2$ can be replaced by any other simply connected domain $\Theta \subset\R^2$ with $C^{1,\beta}$-boundary , $\beta>0$, according to the Riemannian mapping theorem and the well established boundary behaviour of conformal mappings; see for instance \cite{Pommerenke} Chapter 3.
\end{remark}

Returning to Heinz' conjecture mentioned above, we obtain the following

\begin{corollary}
\label{co:h-surface}
Let $H\in L^\infty(\R^3)$ be given and let $u\in W^{1,2}(D^2,\R^3)$ be a solution of (\ref{eq:h-surface}) with continuous boundary trace $u|_{\partial D^2}$. Then there holds $u\in C^{0,\alpha}(D^2)\cap C^0(\overline{D^2})$ for some $\alpha>0$. 
\end{corollary} 

This follows directly from Theorem\,\ref{th:boundreg} by writing (\ref{eq:h-surface}) in the form (\ref{eq:sys}) with
	$$\Omega:=H(u)\begin{pmatrix}
	0 & \nabla^\bot u^3 & -\nabla^\bot u^2\\
	-\nabla^\bot u^3 & 0 & \nabla^\bot u^1\\
	\nabla^\bot u^2 & -\nabla^\bot u^1 & 0
	\end{pmatrix}\in L^2(D^2,so_m\otimes \R^2),$$
where we abbreviated $\nabla^\bot:=(-\partial_{x^2},\partial_{x^1})$ and $u=(u^1,u^2,u^3)$.

Let us emphasize that Theorem\,\ref{th:boundreg} can be applied, more generally, to stationary points of conformally invariant functionals in two dimensions. Having Gr\"uter's \cite{Grueter} characterization in mind, we can give the following geometric description (see e.g.~\cite{Chone} and \cite{Riviere06} for details): Let $\mathcal{N}$ be a smooth manifold embedded into some $\R^m$ such that its second fundamental form $A$ is bounded and let $\omega$ be a $2$-form on $\mathcal N$ of class $C^1$ such that $d\omega$ is bounded on $\mathcal N$. Then any stationary point $u \in W^{1,2}(D^2,\mathcal{N})$ of the (conformally invariant) functional
\begin{equation}
\label{eq:functional} 
 \mathcal F(u)=\int_{D^2}\big[|\nabla u|^2+\omega(u)(\partial_{x^1}u,\partial_{x^2}u)\big]\,dx^1\wedge dx^2
\end{equation}
solves the system 
\begin{equation}
\label{eq:stationary}
 \lap u^i + A_{j,l}^i(u)\nabla u^j \cdot \nabla u^l + \lambda_{j,l}^i(u)\ \partial_{x^1} u^j\ \partial_{x^2} u^l = 0\quad\mbox{in}\ D^2\quad\quad (1 \leq i \leq m).
\end{equation}
Thereby, $\lambda=(\lambda_{j,l}^i)$ is some quantity determind by $d\omega$, which is bounded and skew-symmetric, i.e.~$\lambda_{j,l}^i=-\lambda_{i,l}^j$. The system (\ref{eq:stationary}) can be rewritten into the form \eqref{eq:sys} as presented in \cite{Riviere06}, Theorem I.2. Hence, we proved the following counterpart to a conjecture by S.\,Hilde\-brandt \cite{Hil82}, \cite{Hil83}: 

\begin{corollary}
Under the assumptions just mentioned, any stationary point $u\in W^{1,2}(D^2,\mathcal{N})$ of the functional (\ref{eq:functional}) possessing a continuous boundary trace $u|_{\partial D^2}$ belongs to $C^0(\overline{D^2})$.
\end{corollary}

Let us sketch the \emph{plan of the proof of Theorem\,\ref{th:boundreg}}: 
As already mentioned, we combine ideas by Rivi\`ere, Rivi\`ere-Struwe and Strzelecki. More precisely: We use the technique of Rivi{\`e}re and Struwe in the proof of Theorem 1.1 of \cite{StruweRiviere}, showing that Morrey type estimates hold \textit{below the natural exponent $p=2$} (see Lemma \ref{la:abschfgh}; this idea already appeared in \cite{Strzelecki03}). For this purpose, we apply Rivi\`ere's decomposition result (see Lemma \ref{la:Uhlenbeck} below) and decompose $\Omega\in L^2(D^2,so_m\otimes \R^2)$ such that $\mbox{div}(P^{-1} \nabla u)$ belongs to the Hardy space $\mathcal{H}$ up to a harmless bounded factor; here $P$ denotes an appropriate orthogonal transformation. Then we use that $u$ lies in $BMO$ and the duality between $BMO$ and $\mathcal{H}$ to obtain uniform Morrey type estimates for $p \in (1,2)$. In Section \ref{ssec:harmonicwente} we will recall the definitions of $\mathcal H$ and its dual $BMO$ and we collect some of their properties. 

Once established the Morrey type estimates, we can apply an adapted version of the Dirichlet growth theorem, obtaining an appropriate estimate for the modulus of continuity for $u$ (see Proposition\,\ref{pr:dirichlet} below). Then the desired continuity of a solution $u\in W^{1,2}(D^2,\mathbb R^m)$ of (\ref{eq:sys}) with continuous trace $u|_{\partial D^2}$ results from Lemma\,\ref{la:bound:strzlecki} due to Strzelecki; see Section \ref{strzcont} for a review of its proof.

In \cite{Hajlasz}, Hajlasz, Strzelecki and Zhong proved a similar Morrey-type estimate for a seemingly different system. But by Rivi\`ere's Gauge decomposition our system (\ref{eq:sys}) can be brought into the form of \cite{Hajlasz}; see equation (\ref{eq:system:strz}). Thus it is possible to recover a uniform estimate also by their method.

After finishing the manuscript, T.\,Lamm drew our attention to Rivi\`ere's survey paper \cite{RivOver}, where interior Morrey type estimates were established even for $p=2$, based on Rivi\`ere's conservation law. Note that Rivi\`ere's result can be used to give an alternative proof of Theorem\,\ref{th:boundreg} but no qualitatively better result can be obtained. In addition, the arguments used here and adapted from \cite{StruweRiviere} to establish the Morrey type estimates for $p<2$ seem to be more general in having generalizations to higher dimensional systems. We wish to thank Tobias Lamm for his valuable hint. 

We conclude by fixing some notation for the whole paper: As already mentioned, $\nabla = (\partial_{x^1},\partial_{x^2})^t$ denotes the gradient, while $\nabla^\bot = (-\partial_{x^2}, \partial_{x^1})^t$ denotes the "rotated" gradient. We write $B_r(x)$ for a disc of radius $r>0$ around the center $x\in \R^2$ and $D^2:=B_1(0)$ for the unit disc. We also define the mean value of some function $u$ over $B_r(x)$, 
\[
 (u)_{x,r}:=\mvint_{B_r(x)} u = \frac{1}{\abs{B_{r}(x)}} \int_{B_r(x)} u.
\]
Finally, $so_m\otimes\mathbb F$ and $SO_m\times\mathbb F$ denote the space of $m\times m$-matrices with components in the field $\mathbb F$ and, as usual, we write shortly $so_m:=so_m\otimes\mathbb R$, $SO_m:=SO_m\otimes\mathbb R$.

\vspace{2ex}\noindent
{\bf Acknowledgement.} The second author was partially supported by the RWTH Aachen and Studienstiftung des Deutschen Volkes for which he likes to express his gratitude.


%% file: calculations-1.tex
\section{Proof of Theorem \ref{th:boundreg}}
\subsection{Decompositions}\label{decomp}
We intend to apply the following nonlinear decomposition, which is due to Rivi\`ere and adapts Uhlenbeck's technique in \cite{Uhlenbeck}:
\begin{lemma}[Rivi{\`e}re, 2007]
\label{la:Uhlenbeck}(c.f.~\cite{Riviere06}, lemma A.3)\\
There are $\eps[m] \in (0,1)$ and $\constant[m] > 0$ such that for every $a \in \mathbb R^2$ and $r > 0$ and for every $\Omega \in L^2(B_r(a),so_m\otimes \R^2)$ with
\[
 \Vert \Omega \Vert_{L^2(B_r(a))} \leq \eps[m]
\]
we have the following decomposition: There exist $P \in W^{1,2}(B_r(a),SO_m)$ and $\xi \in W^{1,2}(B_r(a),so_m)$ such that
\[
 \nabla^\bot \xi = P^{-1} \nabla P + P^{-1} \Omega P \mbox{\quad in $B_r(a)$}
\]
is true. In addition, there holds the estimate
\[
 \Vert \nabla \xi \Vert_{L^2(B_r(a))} + \Vert \nabla P \Vert_{L^2(B_r(a))} \leq \constant[m]\ \Vert \Omega \Vert_{L^2(B_r(a))}\mbox{.}
\]
The constants $\constant[m]$ and $\eps[m]$ are independent of $r$ and $a$ as can be seen by shifting and scaling.
\end{lemma}
For the convenience of the reader we will sketch the proof of Lemma \ref{la:Uhlenbeck} in the Appendix. 

In the following, we will disregard the dependence of constants from the image dimension $m\in\mathbb N$, since $m$ will be fixed for the whole paper. In particular, we write $\varepsilon=\eps[m]$ and $C=\constant[m]$ for the constants determined in Lemma\,\ref{la:Uhlenbeck}.

Now let $\Omega\in L^2(D^2,so_m\otimes\mathbb R^2)$ and $e\in L^s(D^2,\mathbb R^m)$, $s>1$, be given as in Theorem \ref{th:boundreg}. Choose $\delta \in (0,\varepsilon)$ to be fixed later and define $R_0=R_0(\delta)\in (0,1)$ such that 
\begin{equation}
\label{eq:omega}
 (1+C) \Vert \Omega \Vert_{L^2(B_{2R_0}(a) \cap D^2)} \leq \delta \mbox{\quad for all $a \in D^2$}
\end{equation}
is satisfied with the constants $\varepsilon,C$ determined in Lemma \ref{la:Uhlenbeck}. Pick $x_0\in D^2$ and $R>0$ such that $R<\min\{1-|x_0|,R_0\}$ holds true. By extending $\Omega$ formally to $0$ out of $D^2$ and applying Lemma \ref{la:Uhlenbeck}, we then obtain
\begin{equation}
\label{eq:bound:decomp}
 \nabla^\bot \xi = P^{-1} \nabla P + P^{-1} \Omega P \mbox{\quad in $B_{2R_0}(x_0)$}
\end{equation}
and
\begin{equation}
\label{eq:bound:decompest}
\Vert \nabla P \Vert_{L^2(B_{2R_0}(x_0))} + \Vert \nabla \xi \Vert_{L^2(B_{2R_0}(x_0))} \leq \delta\mbox{.}
\end{equation}
Now consider a weak solution $u\in W^{1,2}(D^2,\mathbb R^m)$ of 
\[
  -\lap u=\Omega\cdot\nabla u+e\quad\mbox{in}\ D^2
\]
as in Theorem \ref{th:boundreg}. Formula (\ref{eq:bound:decomp}) then yields
\begin{equation}
\label{eq:system:strz}
 - \mbox{div}(P^{-1} \nabla u) = \nabla^\bot \xi \cdot P^{-1} \nabla u + P^{-1}e\quad\mbox{weakly in}\ B_{2R_0}(x_0) \cap D^2.
\end{equation}
Next, let $x_1 \in D^2$ and $\varrho>$ be chosen with $B_{2\varrho}(x_1) \subset B_{R}(x_0)$. Then, a linear Hodge-decomposition\footnote{For arbitrary $\chi\in L^2(D^2,\R^m\otimes\R^2)$, define $f,g\in W^{1,2}_0(B_\varrho(x_1),\R^m)$ as weak solutions of $\lap f=\mbox{div}(\chi)$ and $\lap g=\mbox{curl}(\chi)$. Then $h:=\chi-\nabla f-\nabla^\bot g$ is harmonic in $D^2$.}
 gives us
\begin{equation}
\label{eq:bound:hdec}
 P^{-1} \nabla u = \nabla f + \nabla^\bot g + h \mbox{\quad in $B_\varrho(x_1)$}
\end{equation}
with functions $f,g \in W^{1,2}_0(B_\varrho(x_1),\R^m)$ and a harmonic $h \in L^2(B_\varrho(x_1),\R^m \otimes \R^2)\cap C^\infty(B_\varrho(x_1),\R^m \otimes \R^2)$. In addition, we have
\begin{equation}
\label{eq:bound:lfg}
\begin{split} - \lap f &= -\mbox{div}(P^{-1}\nabla u)=\nabla^\bot \xi\cdot P^{-1} \nabla u  + P^{-1} e \mbox{\quad in $B_\varrho(x_1)$,}\\
- \lap g &= -\mbox{curl}(P^{-1}\nabla u)=\nabla^\bot P^{-1}\cdot \nabla u \mbox{\quad in $B_\varrho(x_1)$.}
\end{split}
\end{equation}
For $r \in (0,\varrho)$ and $p \in (1,2)$ we now can estimate
\begin{equation}
\label{eq:bound:ve1}
 \begin{ma}
  \int_{B_r(x_1)} \abs{\nabla u}^p\ dx &=& \int_{B_r(x_1)} \abs{P^{-1}\nabla u}^p\ dx \\ 
&\overset{\eqref{eq:bound:hdec}}{\leq}& \constant[p] \int_{B_r(x_1)} \big(\abs{h}^p + \abs{\nabla f}^p + \abs{\nabla g}^p\big).
 \end{ma}
\end{equation}
Consequently, Morrey type $L^p$-estimates for $f,g,h$ will yield such estimates for the solution $u$.

\subsection{Morrey type estimates}\label{morrey}
Pick $x_1\in D^2$ and $\varrho>0$ with $B_{2\varrho}(x_1)\subset D^2$. Define
\begin{equation}
\label{eq:bound:vrho}
 v_\varrho := \eta (u - (u)_{x_1,\varrho})
\end{equation}
for the $u\in W^{1,2}(D^2,\mathbb R^m)$ from (\ref{eq:bound:lfg}), where $\eta \in C_0^\infty(B_{\frac{3}{2} \varrho}(x_1))$ is some cut-off function satisfying $0 \leq \eta \leq 1$, $\eta \equiv 1$ on $B_\varrho(x_1)$ and $\abs{\nabla \eta} \leq \frac{C}{\varrho}$. Then we have the following crucial estimates:
\begin{lemma}
\label{la:abschfgh}
Let $f,g\in W^{1,2}_0(B_\varrho(x_1),\R^m)$ be solutions of (\ref{eq:bound:lfg}) for some given $P \in W^{1,2}(B_\varrho(x_1),SO_m)$, $\xi \in W^{1,2}(B_\varrho(x_1),so_m)$ satisfying (\ref{eq:bound:decompest}) with small $\delta>0$, as well as some $u\in W^{1,2}(\Omega,\R^m)$ and $e\in L^s(B_1(x_1),\R^m)$ with $s>1$. Then, for every $p \in (1,2)$, there is a constant $C \equiv \constant[p,s]$ such that
\begin{equation}
 \label{eq:bound:nfgp}
 \Vert \nabla f \Vert_{L^p(B_\varrho(x_1))} + \Vert \nabla g \Vert_{L^p(B_\varrho(x_1))} \leq C\ \varrho^{\frac{2}{p} -1}  \delta [v_\varrho]_{BMO} + C\ \varrho^{1 + \frac{2}{p}-\frac{2}{s}} \Vert e \Vert_{L^s(B_{\varrho}(x_1))}.
\end{equation}
Furthermore, for any harmonic $h\in L^p(B_\varrho(x_1))$ and any $r\in(0,\varrho]$, there holds
\[
 \int_{B_r(x_1)}\abs{h}^p \leq \constant[p] \left ( \frac{r}{\varrho} \right )^2 \int_{B_\varrho(x_1)} \abs{h}^p.
\]
\end{lemma}
\begin{proof}
Fix a number $p \in (1,2)$ and let $q = \frac{p}{p-1} > 2$ be the conjugated exponent of $p$. According to $f\in W^{1,2}_0(B_\varrho(x_1),\mathbb R^m)$ we then infer\footnote{as $F := \frac{\nabla f \abs{\nabla f}^{p-2}}{\Vert \nabla f \Vert^{p-1}_{L^p}} \in L^q(B_\varrho(0))$ and $\Vert F \Vert_{L^q} = 1$ and we can decompose every $q$-integrable $\tilde{F} = \nabla \varphi + G$, where $div(G) = 0$ and $\varphi \in W^{1,q}_0$. Furthermore we have $\Vert \nabla \varphi \Vert_{L^q} \leq C \Vert f \Vert_{L^q}$; see the proof of Wente's inequality, Theorem \ref{th:wente}, in Subsection \ref{ssec:harmonicwente} for more details.}
\[
 \Vert \nabla f \Vert_{L^p(B_\varrho(x_1))} \leq C \sup_{\varphi \in C_0^\infty(B_\varrho(x_1))\atop \Vert \varphi \Vert_{W^{1,q}} \leq 1} \int_{B_\varrho(x_1)} \nabla f \cdot \nabla \varphi.
\]
As $q > 2$ we have the embedding $W^{1,q}(B_\varrho(x_1)) \hookrightarrow C^{1-\frac{2}{q}}(\overline{B_\varrho(x_1)})$. Hence, for every $x,y \in B_\varrho(x_1)$, $x \neq y$, and every $\varphi \in C_0^\infty(B_\varrho(x_1))$, we get
\begin{equation}
\label{eq:bound:absvp}
 \begin{split} 
  \abs{\varphi(x)} &\leq \abs{\varphi(x) - \varphi(y)} + \abs{\varphi(y)}\\
&\leq \frac{\abs{\varphi(x)-\varphi(y)}}{\abs{x-y}^{1-\frac{2}{q}}} \varrho^{1-\frac{2}{q}} + \abs{\varphi(y)}\\
&\leq \varrho^{1-\frac{2}{q}}\,\mbox{Hol}_{1-\frac{2}{q},B_\varrho(x_1)} \varphi + \abs{\varphi(y)}\\
&\leq \varrho^{1-\frac{2}{q}}\ \constant[q]\ \Vert \nabla \varphi \Vert_{L^q(B_\varrho(x_1))} + \abs{\varphi(y)}. 
\end{split}
\end{equation}
Here we used
\[
 \begin{split}
  \mbox{Hol}_{1-\frac{2}{q},B_\varrho(x_1)} \varphi &:=\sup_{x,y \in B_\varrho(x_1)} \frac{\varphi(x)- \varphi(y)}{\abs{x-y}^{1-\frac{2}{q}}}\\
&= \varrho^{\frac{2}{q}-1} \sup_{\tilde{x},\tilde{y} \in B_1(0)} \frac{\varphi(\varrho \tilde{x}+x_1)-\varphi(\varrho \tilde{y}+x_1)}{\abs{\tilde{x}-\tilde{y}}^{1-\frac{2}{q}}}\\
&\leq \varrho^{\frac{2}{q}-1}\ \constant[q] \Vert \nabla \varphi(\varrho (\cdot)+x_1) \Vert_{L^q(B_1(0))}\\[1ex]
&=\constant[q] \Vert \nabla \varphi \Vert_{L^q(B_\varrho(x_1))}.
 \end{split}
\]
As $\varphi \in C_0^\infty(B_\varrho(x_1))$, taking the infimum over all $y \in B_\varrho(x_1)$ on the right hand side of \eqref{eq:bound:absvp}, we obtain (remember that $q$ is the H\"older-conjugate of $p$)
\begin{equation}
\label{eq:bound:sr30}
 \Vert \varphi \Vert_{L^\infty(B_\varrho(x_1))} \leq \constant[p]\ \varrho^{\frac{2}{p} - 1}\ \Vert \nabla \varphi \Vert_{L^q(B_\varrho(x_1))}.
\end{equation}
Furthermore, by H\"older's inequality we have
\begin{equation}
\label{eq:bound:nphi}
 \Vert \nabla \varphi \Vert_{L^2(B_\varrho(x_1))} \leq \constant[p]\ \varrho^{\frac{2}{p}-1} \Vert \nabla \varphi \Vert_{L^q(B_\varrho(x_1))}.
\end{equation}
Now, our $\varphi \in C_0^\infty(B_\varrho(x_1))$ with $\Vert \varphi \Vert_{W^{1,q}} \leq 1$ is an admissible testfunction for the equation of $f$ in \eqref{eq:bound:lfg} and we compute\footnote{\label{ft:bound:nfnp}The following calculations require some care: They work if we assume $\xi$ to be smooth. If this is not the case the integrals
\[
 \int \nabla^\bot \xi\cdot \nabla (P^{-1} u \varphi),\quad \int \nabla^\bot \xi\cdot \nabla (P^{-1} \varphi)\ u
\]
are a priori not well-defined. Indeed, $(\nabla P^{-1}) u$ is not necessarily in $L^2$ and integration by parts could fail. To overcome this, we approximate $\xi$ by smooth $\xi_k$; then the calculations are the same where $\xi$ is replaced by $\xi_k$ - up to an additional term
\[
 \int_{B_\varrho(x_1)} \nabla^\bot (\xi_k - \xi) P^{-1}\ \nabla \varphi\ u \xrightarrow{k \to \infty} 0.
\]}
\[
\begin{ma}
 \int_{B_\varrho(x_1)}\!\!\! \nabla f \cdot \nabla \varphi &\overset{\eqref{eq:bound:lfg}}{=}& \int_{B_\varrho(x_1)} \nabla^\bot \xi \cdot P^{-1} \nabla u\ \varphi + \int_{B_\varrho(x_1)} P^{-1} e \varphi\\
&=&  \!-\int_{B_\varrho(x_1)} \nabla^\bot \xi\cdot \nabla (P^{-1}\varphi) u + \int_{B_\varrho(x_1)} P^{-1} e \varphi\\
&=&  \!-\int_{B_\varrho(x_1)}\!\! \nabla^\bot \xi\cdot \nabla (P^{-1}\varphi) (u-(u)_{x_1,\varrho}) + \int_{B_\varrho(x_1)}\!\! P^{-1} e \varphi.
\end{ma}
\]
Because of $\varphi \in C_0^\infty(B_\varrho(x_1))$, we can smuggle our $\eta$ from \eqref{eq:bound:vrho} into this equation. The duality of Hardy- and BMO-space, Theorem \ref{th:bmoh}, then implies
\[
\begin{ma}
 \int_{B_\varrho(x_1)} \!\!\!\nabla f \cdot \nabla \varphi  &=& -\int_{B_\varrho(x_1)}\!\nabla^\bot \xi\cdot \nabla (P^{-1}\varphi)\ \eta (u-(u)_{x_1,\varrho}) + \int_{B_\varrho(x_1)}\!\! P^{-1} e \varphi\\
&\leq& C\Vert \nabla^\bot \xi \cdot \nabla (P^{-1} \varphi)\Vert_{\mathcal {H}^1}\ \left [ \eta (u - (u)_{x_1,\varrho}) \right ]_{BMO} \\
&& +  \Vert \varphi\  e \Vert_{L^1(B_\varrho(x_1))}.
\end{ma}
\]
By H\"older's inequality we know
\[
 \Vert e \Vert_{L^1(B_\varrho(x_1))} \leq \constant[s]\ \Vert e \Vert_{L^s(B_{\varrho}(x_1))}\  \varrho^{2-\frac{2}{s}},
\]
and, consequently,
\[
 \Vert \varphi e \Vert_{L^1(B_\varrho(x_1))} \overset{\eqref{eq:bound:sr30}}\leq \constant[s,p]  \Vert \nabla \varphi \Vert_{L^q(B_\varrho(x_1))}\ \varrho^{1+\frac{2}{p} - \frac{2}{s}} \Vert e \Vert_{L^s(B_{\varrho}(x_1))}.
\]
Next, we consider the Hardy-term: By extending\footnote{By using an extension of some $f \in W^{1,2}(\disk)$ and scaling the inequality
\[
 \Vert \nabla f \Vert_{L^2(\R^2)} \leq C \Vert f - (f)_{0,1} \Vert_{W^{1,2}(D^2)} \leq C \Vert \nabla f \Vert_{L^2(D^2)}
\]
we see, that the norm of the extension-operator for balls is independent of the radius.} $\xi-(\xi)_{x_1,\varrho}$ and $P^{-1} \varphi-(P^{-1} \varphi)_{x_1,\varrho}$ to $W^{1,2}(\R^2,\R^{n\times n})$-mappings, we deduce from Theorem \ref{th:CLMS}  (remember that $P,\ P^{-1}$ have their values in $SO_m$ a.e.~and thus are bounded):
\[
\begin{ma} 
\hspace*{-2ex}\Vert \nabla^\bot \xi \cdot \nabla (P^{-1} \varphi)\Vert_{\mathcal {H}^1} &\hspace*{-1.8ex}\leq&\hspace*{-1.8ex} C \Vert \nabla (\xi-(\xi)_{x_1,\varrho}) \Vert_{L^2(\R^2)} \Vert \nabla (P^{-1} \varphi-(P^{-1} \varphi)_{x_1,\varrho}) \Vert_{L^2(\R^2)}\\
&\hspace*{-3ex}\leq\hspace*{-3ex}& \,C \Vert \nabla \xi \Vert_{L^2(B_{\varrho}(x_1))}\ \Vert \nabla (P^{-1} \varphi) \Vert_{L^2(B_\varrho(x_1))}\\
&\hspace*{-3ex}\overset{\eqref{eq:bound:decompest}}{\leq}\hspace*{-3ex}& \,C\, \delta (\delta \Vert \varphi \Vert_{L^\infty(B_\varrho(x_1))} + \Vert \nabla \varphi \Vert_{L^2(B_\varrho(x_1))} )\\
&\hspace*{-3ex}\overset{\eqref{eq:bound:sr30}\atop \eqref{eq:bound:nphi}}{\leq}\hspace*{-3ex} &  \,C\, \varrho^{\frac{2}{p}-1} \delta \Vert \nabla \varphi \Vert_{L^q(B_\varrho(x_1))}.
\end{ma}
\]
Putting everything together and using $\Vert \varphi \Vert_{W^{1,q}} \leq 1$ as well as the definition (\ref{eq:bound:vrho}) of $v_\varrho$, we proved the estimate for $f$ in \eqref{eq:bound:nfgp}.

For $g$ we calculate\footnote{Here we have the same problem as for the calculations of $f$, which we solve by approximation of $P$, cf. footnote \ref{ft:bound:nfnp}, page \pageref{ft:bound:nfnp}.}
\[
\begin{ma}
 \int_{B_\varrho(x_1)} \nabla g \cdot \nabla \varphi &\overset{\eqref{eq:bound:lfg}}=& \int_{B_\varrho(x_1)} \nabla^\bot 
P^{-1}\cdot \nabla u\ \varphi\\
&=& \int_{B_\varrho(x_1)} \nabla^\bot P^{-1}\cdot \nabla \varphi\ [\eta(u - (u)_{x_0,\varrho})]\\
&\leq& \Vert \nabla^\bot P^{-1} \cdot \nabla \varphi \Vert_{\mathcal{H}}\ [v_\varrho]_{BMO}\\
&\leq& C \Vert \nabla P^{-1} \Vert_{L^2(B_\varrho(x_1))}\ \Vert \nabla \varphi \Vert_{L^2(B_\varrho(x_1))}\ [v_\varrho]_{BMO}\\
&\overset{\eqref{eq:bound:nphi}}{\leq}&C \Vert \nabla P^{-1} \Vert_{L^2(B_\varrho(x_1))}\ \varrho^{\frac{2}{p}-1} \Vert \nabla \varphi \Vert_{L^q(B_\varrho(x_1))}\ [v_\varrho]_{BMO}\\
&\leq&C \Vert \nabla P^{-1} \Vert_{L^2(B_\varrho(x_1))}\ \varrho^{\frac{2}{p}-1}\ [v_\varrho]_{BMO}.\\
\end{ma}
\]
The constant in the Hardy-estimate is independent of $\varrho$, as one can see by scaling. Hence, we also obtained the estimate for $g$ in \eqref{eq:bound:nfgp}.

Finally, we establish the estimate for the harmonic term $h$. Estimating as in Theorem\,2.1 on p.~78 of Giaquinta's monograph \cite{GiaquintaMultipleIntegrals} and applying the embedding $W^{2,p}\hookrightarrow L^\infty$ as well as $L^p$-theory, we find: For any $h \in L^p(B_1(0))$, $p > 1$, with $\lap h = 0$ in $B_1(0)$ and for any $\gamma\in(0,\frac12)$ holds
\[
\begin{ma}
 \int_{B_\gamma(0)} \abs{h}^p &\leq& \constant[p] \gamma^2 \Vert h \Vert_{L^\infty(B_{\frac{1}{2}}(0))}^p\\
&\leq& \constant[p] \gamma^2 \Vert h \Vert_{W^{2,p}(B_{\frac{1}{2}}(0))}^p\\
&\leq& \constant[p] \gamma^2 \Vert h \Vert_{L^p(B_1(0))}^p.
\end{ma}
\]
Of course, this result remains valid also for $1 \geq \gamma \geq \frac{1}{2}$. Hence, upon shifting the inequality to some $x_1$ and scaling, we infer
\[
 \int_{B_r(x_1)} \abs{h}^p \leq \constant[p] \left (\frac{r}{\varrho} \right )^2 \int_{B_\varrho(x_1)} \abs{h}^p.
\]
for every harmonic $h \in L^p(B_\varrho(x_1))$, $p > 1$, and for every $0 < r \leq \varrho$. This completes the proof.
\qed
\end{proof}

\subsection{Completion of the proof of Theorem \ref{th:boundreg}}\label{complete}
Let the assumptions of Theorem \ref{th:boundreg} be satisfied, that is: Consider a weak solution $u\in W^{1,2}(\Omega,\R^m)$ of 
\begin{equation}
\label{eq:sys:again}
  -\lap u=\Omega\cdot\nabla u+e\quad\mbox{in}\ D^2
\end{equation}
with continuous trace $u|_{\partial D^2}$. Here, $\Omega\in L^2(D^2,so_m\otimes\R^2)$ and $e\in L^s(D^2,\R^m)$, $s>1$, are given, and w.l.o.g.~we may assume $s\in(1,\frac43)$. 

We return to the situation described in Subsection \ref{decomp}: Choosing an arbitrary $\delta\in(0,\eps[m]]$ and  $R_0=R_0(\delta)\in(0,1)$ suitably such that (\ref{eq:omega}) is fulfilled, we pick $x_0\in D^2$ and $R>0$ with $R<\min\{1-|x_0|,R\}$. For any $x_1\in D^2$ and $\varrho>0$ with $B_{2\varrho}(x_1)\subset B_R(x_0)$ we then found functions $f,g\in W^{1,2}_0(B_\varrho(D^2),\R^m)$, which solve (\ref{eq:bound:lfg}) with $P,\xi$ from Lemma \ref{la:Uhlenbeck}, and some harmonic function $h\in L^2(B_\varrho(x_1),\R^m\otimes\R^2)$ such that the estimate (\ref{eq:bound:ve1}) is fulfilled for any $r\in(0,\varrho)$ and any $p\in(1,2)$. Combining this inequality with Lemma \ref{la:abschfgh} from Section \ref{morrey}, we arrive at 
\[
 \begin{ma}
  \int_{B_r(x_1)} \abs{\nabla u}^p &\leq& \constant[p] \left (\frac{r}{\varrho }\right )^2 \int_{B_\varrho (x_1)} \abs{h}^p + \constant[p] \int_{B_\varrho (x_1)} (\abs{\nabla f}^p + \abs{\nabla g}^p)\\
&\overset{\eqref{eq:bound:hdec}}{\leq}& \constant[p] \left (\frac{r}{\varrho } \right )^2 \int_{B_\varrho (x_1)} \abs{\nabla u}^p + \constant[p] \int_{B_\varrho (x_1)} (\abs{\nabla f}^p + \abs{\nabla g}^p)\\
&\overset{\eqref{eq:bound:nfgp}}{\leq}& \constant[p] \left (\frac{r}{\varrho } \right )^2 \int_{B_\varrho (x_1)} \abs{\nabla u}^p + \constant[p] \delta\ \varrho ^{2-p}\ [v_\varrho]^p_{BMO}\\
&&\quad + \constant[p,s] \varrho^{2-p + 2p - \frac{2p}{s}}\ \Vert e \Vert_{L^s(B_\varrho(x_1))}^p.
 \end{ma}
\]
Multiplying this by $r^{p-2}$ and defining
\[
 J_p(a,r;u) := \frac{1}{r^{2-p}} \int_{B_r(a)} \abs{\nabla u}^p,
\]
\[
\mathcal{M}_p (a,r;u) := \sup_{z \in B_r(a),\atop \varrho < r - \abs{a-z}} \frac{1}{\varrho^{2-p}} \int_{B_\varrho(z)} \abs{\nabla u}^p,
\]
we infer
\[
\begin{split}
 J_p(x_1,r;u) \leq& \ \constant[p] \left (\frac{r}{\varrho}\right )^p J_p(x_1,\varrho;u) + \constant[p] \delta \left (\frac{r}{\varrho} \right )^{p-2} [v_\varrho]^p_{BMO}\\ 
&\quad + \constant[p,s] \left (\frac{r}{\varrho}\right )^{p-2}\ \varrho^{2p(1-\frac{1}{s})}\ \Vert e \Vert_{L^s(B_\varrho(x_1))}^p
\end{split}
\]
for all $0 < r < \varrho$ and $x_1\in D^2$ with $B_{2\varrho}(x_1) \subset B_R(x_0)$. In order to exploit this last relation, we have to estimate $[v_\varrho]_{BMO}$ appropriately. This can be done by exactly the same calculations as in Step 5 of Strzlecki's article \cite{Strzelecki03}:
\begin{proposition}
\label{pr:vrBMO}
There is a constant $\constant[p]$ such that
\[
 [v_\varrho]_{BMO} \leq \constant[p]\ \mathcal{M}_p(x_1,2\varrho;u)^{\frac{1}{p}}
\]
for all $x_1\in D^2$ and $\varrho > 0$ with $B_{2\varrho}(x_1) \subset D^2$.
\end{proposition}
Now, pick some $\gamma < 1$ to be fixed later and set $r := \gamma \varrho$. Then, for all $x_1 \in D^2$ and $\varrho > 0$ with $B_{2\varrho}(x_1) \subset B_{R}(x_0)$, we conclude
\[
\begin{split}
 J_p(x_1,\gamma \varrho;u) \leq& \ \constant[p]\ \gamma^p\ (1 + \delta \gamma^{-2})\ \mathcal{M}_p(x_0,R;u)\\
&\quad + \constant[p,s] \gamma^{p-2} \varrho^{2p(1-\frac{1}{s})}\ \Vert e \Vert_{L^s(B_R(x_0))}^p.
\end{split}
\]
The constant $\constant[p]$ is independent of $R_0$ and hence of $\delta$. We choose $\gamma < 1$ small enough to ensure $\constant[p]\ \gamma^p < \frac{1}{4}$. Setting $\delta :=\min(\gamma^2,\varepsilon)$, we get the estimate
\[
 J_p(x_1,\gamma \varrho;u) \leq \frac{1}{2} \mathcal{M}_p(x_0,R;u) + \constant[s,p,\gamma] R^{2p(1-\frac{1}{s})} \ \Vert e \Vert_{L^s(B_R(x_0))}^p
\]
for all $x_1 \in D^2$ and $\varrho > 0$ with $B_{2\varrho}(x_1) \subset B_R(x_0)$. This can be written equivalently as 
\[
 J_p(x_1,r;u) \leq \frac{1}{2} \mathcal{M}_p(x_0,R;u) + \constant[s,p,\gamma]\ R^{2p(1-\frac{1}{s})} \ \Vert e \Vert_{L^s(B_R(x_0))}^p
\]
for all $x_1\in D^2$ and $r>0$ with $B_{\frac{2r}\gamma}(x_1)\subset B_R(x_0)$. Due to $\frac\gamma2<1$, this holds true especially for all $x_1\in D^2$ and $r>0$ with $B_r(x_1)\subset B_{\frac\gamma2 R}(x_0)$. Taking the supremum over all those $x_1,r$, we arrive at
\begin{equation}\label{eq:bound:morreymitgamma}
 \mathcal{M}_{p} \big(x_0,\frac{\gamma}{2} R; u\big) \leq \frac{1}{2} \mathcal{M}_{p} (x_0,R; u) + \constant[s,p,\gamma]\ R^{2p(1-\frac{1}{s})} \ \Vert e \Vert_{L^s(B_R(x_0))}^p
\end{equation}
for all $x_0\in D^2$ and $R \in (0,R_0]$ with $B_{R}(x_0) \subset D^2$. Note that $l := 2p(1-\frac{1}{s})\in(0,1)$ is true, according to $p \in (1,2)$ and $s \in (1,\frac{4}{3})$.

We conclude by standard-tricks: Set $\tilde{\gamma} := \frac{\gamma}{2}\in(0,\frac12)$ and pick some $0<r < R < \min(R_0,1-\abs{x_0})$. Let $i \in \N_0$ be chosen such that
\[
 \tilde{\gamma}^{i+1} R< r \leq \tilde{\gamma}^i R
\]
is satisfied. Furthermore, let $\theta \in (0,1)$ be defined to fulfill
\[
 \tilde{\gamma}^\theta = \frac{1}{2}.
\]
Hence, the monotonicity of the mapping $r\mapsto\mathcal{M}(x_0,r;u)$ implies
\[
 \begin{ma}
  \!\!\!\mathcal{M}_p(x_0,r;u)\hspace*{-3ex} &\leq& \hspace*{-2ex}\mathcal{M}_p(x_0,\tilde{\gamma}^i R; u)\\
&\overset{\eqref{eq:bound:morreymitgamma}}{\leq}& \hspace*{-2ex}\left (\frac{1}{2}\right )^i \mathcal{M}_p(x_0,R;u) + C\sum_{k=1}^{i} (\tilde{\gamma}^{i-k} R)^l \left (\frac{1}{2} \right )^k\ \Vert e \Vert_{L^s(B_R(x_0))}^p\\
&\overset{\tilde{\gamma} < \frac{1}{2}\atop l > 0}{\leq}& \hspace*{-2ex}\left (\frac{1}{2}\right )^i \mathcal{M}_p(x_0,R;u) + R^l \left ( \frac{1}{2} \right )^{il} \frac{1}{1-\left (\frac{1}{2} \right )^{1-l}}\ C\Vert e \Vert_{L^s(B_R(x_0))}^p\\
&\overset{l < 1\atop R < 1}{=}& \hspace*{-2ex} 2 \left (\tilde{\gamma}^\theta \right )^{i+1} \mathcal{M}_p(x_0,R;u) + \constant[s,p,\gamma,l] \left ( \tilde{\gamma}^\theta \right )^{(i+1)l}\ \Vert e \Vert_{L^s(B_R(x_0))}^p\\
&=& \hspace*{-2ex}2 \left (\tilde{\gamma}^{i+1} \right )^\theta \mathcal{M}_p(x_0,R;u) + \constant[s,p,\gamma] \left ( \tilde{\gamma}^{i+1} \right )^{\theta l}\ \Vert e \Vert_{L^s(B_R(x_0))}^p\\
&\leq& \hspace*{-2ex} 2 \left (\frac{r}{R} \right )^\theta \mathcal{M}_p(x_0,R;u) + \constant[s,p,\gamma] \left (\frac{r}{R} \right )^{\theta l}\ \Vert e \Vert_{L^s(B_R(x_0))}^p\\
&\overset{l < 1\atop r \leq R}{\leq}& \hspace*{-2ex} 2 \left (\frac{r}{R} \right )^{\theta l} \mathcal{M}_p(x_0,R;u) + \constant[s,p,\gamma] \left (\frac{r}{R} \right )^{\theta l}\ \Vert e \Vert_{L^s(B_R(x_0))}^p\\
 \end{ma}
\]
for all $x_0\in D^2$ and $0 < r \leq R \leq \min\{1-\abs{x_0},R_0\}$. 
Setting $\mu := \theta l \in (0,1)$ we finally conclude
\begin{equation}\label{eq:bound:dirgrowth}
 \mathcal{M}_p(x_0,r;u) \leq \constant[s,p,\gamma] \Big[\mathcal{M}_p(x_0,R;u) + \Vert e \Vert_{L^s(B_R(x_0))}^p\Big] \left (\frac{r}{R} \right )^{\mu}
\end{equation}
for all $x_0\in D^2$ and $0 < r < R < R_0$ with $B_R(x_0) \subset D^2$.

The standard Dirichlet growth theorem now implies the claimed interior regularity $u \in C^{0,\alpha}(D^2)$ for some $\alpha \in (0,1)$. To derive boundary regularity we need the following variant, which follows by Morrey's technique in \cite{Morrey} Theorem 3.5.2, p. 79, for the Dirichlet growth theorem on the ball:
\begin{proposition}[Dirichlet growth theorem]\label{pr:dirichlet}
There is a constant $C$ such that, for all $\varrho \in (0,R_0)$, $a \in D^2$ with  $B_{\varrho}(a) \subset D^2$ and for any solution $u\in W^{1,2}(D^2)$ of (\ref{eq:bound:dirgrowth}), the inequality
\begin{equation}
\label{eq:modulus}
 \abs{u(x)-u(y)} \leq C\frac{p}{\mu} (\Vert \nabla u \Vert_{L^2(B_{\varrho}(a))} + \Vert e \Vert_{L^s(B_{\varrho}(a))} ),\quad x,y\in B_{\frac\varrho2}(a)
\end{equation}
holds true.
\end{proposition}

For convenience, we sketch the proof of Proposition \ref{pr:dirichlet} in Subsection \ref{dirichlet}.

Now, having the estimate (\ref{eq:modulus}) for the modulus of continuity for our solution $u\in W^{1,2}(D^2)$ of (\ref{eq:sys:again}) in mind and assuming the continuity of $u|_{\partial D^2}$, the desired global regularity $u\in C^0(\overline{D^2},\R^m)$ follows from the following lemma by Strzelecki:

\begin{lemma}[Strzelecki, 2003]
\label{la:bound:strzlecki}
(c.f. \cite{Strzelecki03}, lemma 3.1)\\
Let $u \in W^{1,2}(D^2,\R^m) \cap C^{0}(D^2,\R^m)$. Assume that there are $R_0 > 0$ and a mapping $F:\ D^2 \times (0,R_0)\to(0,+\infty)$
such that we have  
\begin{equation}
\label{eq:strz:osc}
 \abs{u(x)-u(y)} \leq F(a,\varrho) \mbox{\quad for all $x,y \in B_{\frac{\varrho}{2}}(a)$}
\end{equation}
for any $\varrho \in (0,R_0)$, $a \in D^2$ with $B_\varrho(a) \subset D^2$. If $F(\cdot,\varrho) \xrightarrow{\varrho \to 0} 0$ uniformly in $D^2$ and if the trace of $u$ on $\partial D^2$ is continuous, then we find $u \in C^0(\overline{D^2},\R^m)$.
\end{lemma}

We recall the proof of this lemma in Subsection \ref{strzcont}. The proof of Theorem \ref{th:boundreg} is completed.\qed\\

%% file: appendix-1.tex
\renewcommand{\thesection}{A}
\renewcommand{\thesubsection}{A.\arabic{subsection}}
\section{Appendix}
For the convenience of the reader we will first state some results from harmonic analysis and, as a corollary, part of Wente's inequality, which we will use afterwards to sketch the proof of the Uhlenbeck-Rivi\`ere decomposition of some skew-symmetric $\Omega$. In accordance with their applications in the present paper, all results are stated on two-dimensional discs, except for the definitions and basic properties of Hardy- and BMO-spaces.  Nevertheless, some results extend in their spirit to higher dimensions.

\subsection{Some facts from Harmonic Analysis and Wente's Inequality}\label{ssec:harmonicwente}
We start with the definitions of BMO and the Hardy-space $\mathcal{H}$. For more details and proofs we refer, e.g., to Stein's monograph \cite{Stein93}. For applications of Hardy spaces to PDE theory the interested reader may consider also Semmes' article \cite{Semmes94}.
\begin{definition}[BMO and Hardy-space]\label{def:bmoh}
Let $\mathcal{T}$ denote the set of testfunctions $\phi \in C_0^\infty(B_1(0))$ with $\abs{\nabla \phi} \leq 1$ everywhere in $B_1(0)$. 
Define the \emph{Hardy space $\mathcal H$} as the space of all functions $f \in L^1(\R^n)$ having their associated maximal function
\[
 f^*(x) := \sup_{t > 0}\ \sup_{\phi \in \mathcal{T}} \abs{\int_{\R^n} \frac{1}{t^n}\ \phi \left ( \frac{x-y}{t} \right ) f(y)\ dy }
\]
in $L^1(\R^n)$. The norm is
\[
 \Vert f \Vert_{\mathcal{H}} := \Vert f^* \Vert_{L^1(\R^n)}.
\]
The \emph{space of bounded mean oscillation BMO }is the space of all $f \in L^1_{loc}(\R^n)$ such that
\[
 [f]_{BMO} = \sup_{x \in \R^n\atop r > 0} \mvint_{B_r(x)} \abs{f- (f)_{x,r}} < \infty
\]
is true with
\[
 (f)_{x,r} = \mvint_{B_r(x)} f.
\]
\end{definition}
Motivated by the results of \cite{Mueller90}, Coifman, Lions, Meyer and Semmes proved in \cite{CLMS} the following
\begin{theorem}[Hardy spaces and div-curl-terms]\label{th:CLMS}
Let $1<p,q<\infty$ with $\frac1p+\frac1q=1$ be chosen. Let $A \in L^p(\R^n,\R^n)$ and $B \in L^q(\R^n,\R^n)$ be weak solutions of
\[
 div(A) = 0 \quad \mbox{and} \quad curl(B) = 0\quad\mbox{in}\ \R^n.
\]
Then we have $A\cdot B \in \mathcal{H}$ and the estimate
\[
 \Vert A \cdot B \Vert_{\mathcal{H}} \leq \Vert A \Vert_{L^p(\R^n)}\ \Vert B \Vert_{L^q(\R^n)}
\]
is true.
\end{theorem}%
The following duality-like theorem was obtained first in \cite{FS72}:
\begin{theorem}[BMO-Hardy-duality]\label{th:bmoh}
There exists a constant $\constant[n]$ depending only on the dimension 
$n$, such that for every smooth $f \in BMO(\R^n)$ and $g \in \mathcal{H}(\R^n)$ the following inequality holds
\[
 \abs{\int_{\R^n} fg } \leq \constant[n]\ [f]_{BMO}\ \Vert g \Vert_{\mathcal {H}}.
\]
\end{theorem}

\begin{theorem}[Wente's inequality]\label{th:wente} (c.f. \cite{Wente69}, \cite{Tartar85}, \cite{BrC84})\\
Let $a \in W^{1,2}(D^2)$, $b \in W^{1,p}(D^2)$ be given with some $p\in(1,\infty)$ and let $u \in W^{1,2}(D^2)$ be a weak solution of
\begin{equation}
\label{eq:wr:pdeu1}
\begin{cases}
  -\lap u = \nabla a \cdot \nabla^\bot b\quad &\mbox{in}\ D^2,\\[1ex]
   u = 0 \quad &\mbox{on}\ \partial D^2.
\end{cases}
\end{equation}
Then $u$ belongs to $W^{1,p}(D^2)$ and we have the inequality
\[
 \Vert \nabla u \Vert_{L^p(D^2)} \leq \constant[p] \Vert \nabla a \Vert_{L^2(D^2)}\ \Vert \nabla b \Vert_{L^p(D^2)}.
\]
\end{theorem}
\begin{proof}
The theorem follows by compactness, if we can prove it for 
\[
 a \in C^\infty(\overline{D^2}),\quad \mvint_{D^2} a = 0,\quad b \in C^\infty(\overline{D^2}).
\]
Furthermore, we assume $a$ and $b$ to be extended to functions with compact support in $W^{1,2}(\R^2)$ and $W^{1,p}(\R^2)$, respectively. Let $q = \frac{p}{p-1}$ be the conjugated exponent of $p$. Writing $X = C_0^\infty(D^2,\R^2)$, we calculate
\[
 \Vert \nabla u \Vert_{L^p(D^2)} = \int_{D^n} \nabla u \frac{\abs{\nabla u}^{p-2} \nabla u}{\Vert \nabla u \Vert_{L^p(D^2)}^{p-1}} \leq \sup_{F \in X\atop \Vert F \Vert_{L^q(D^2)} \leq 1} \int_{D^2} \nabla u \cdot F.
\]
By linear Hodge decomposition, we can split any $F \in X$ into
\[
 F = \nabla \varphi + h,
\]
where $\varphi \in W^{1,2}_0(D^2)$ and $h \in L^2(D^2)$ satisfies
\[
 \int_{D^2} \nabla u \cdot h = 0.
\]
By $L^q$-Theory we have\footnote{\label{ft:lqtheory}Here we use $\Vert g \Vert_{W^{1,p}_0} \leq C \Vert \lap g \Vert_{(W^{1,p}_0)^*}$ which is true for $p \geq 2$ (see for example \cite{GiaquintMartinazziaRegularity}, Theorem 7.1) and which we derive for $p \in (1,2)$ by setting $\Vert g \Vert_{W^{1,p}_0} \leq C \sup_{F \in L^q\atop \Vert F \Vert_{L^q} \leq 1} \int \nabla g \cdot F$. Such $F$ can be decomposed in $\nabla \varphi$ for $\varphi \in W^{1,q}_0$ and some divergence free term, and by the estimates for $q > 2$ we have $\Vert \nabla \varphi \Vert_{L^q} \leq C \Vert F \Vert_{L^q}$.} 
\[
 \Vert \nabla \varphi \Vert_{L^q(D^2)} \leq \constant[q] \Vert F \Vert_{L^q(D^2)}.
\]
Hence, we arrive at
\[
 \Vert \nabla u \Vert_{L^p(D^2)} \leq \constant[q]\ \sup_{\varphi \in Y, \atop \Vert \nabla \varphi \Vert_{L^q(D^2)} \leq 1} \int_{D^2} \nabla u \cdot \nabla \varphi,
\]
where we abbreviated $Y = C_0^\infty(D^2)$. Applying the BMO-Hardy-Duality, Theorem \ref{th:bmoh}, to \eqref{eq:wr:pdeu1}, and then using the extension operator, H\"older- and Poincar\'{e} inequality, we obtain for any $\varphi \in Y$: 
\[
\begin{split}
 \int_{D^2} \nabla u\cdot \nabla \varphi &= \int_{D^2} \nabla a \cdot \nabla^\bot b\ \varphi\\
  &= -\int_{D^2} \nabla a \cdot \nabla^\bot \varphi\ b\\
  &= -\int_{\R^2} \nabla a \cdot \nabla^\bot \varphi\ (b-\mvint_{D^2} b)\\
  &\leq C \Vert \nabla (a-\mvint_{D^2} a) \cdot \nabla^\bot (\varphi-\mvint_{D^2} \varphi) \Vert_{\mathcal{H}}\ [b]_{BMO}\\
  &\leq C \Vert \nabla a \Vert_{L^p(D^2)}\ \Vert \nabla \varphi \Vert_{L^q(D^2)}\ \Vert \nabla b \Vert_{L^2(D^2)},
\end{split}
\]
which completes the proof.\qed
\end{proof}
It is clear, that this type of proof does extend to higher dimensions as well as to the case of homogeneous Neumann boundary data.

\subsection{Decomposition of real skew-symmetric Matrices}
We sketch here the \emph{proof of Lemma \ref{la:Uhlenbeck}}. This result has been proved by Rivi\`ere in \cite{Riviere06}, adapting the techniques by Uhlenbeck, who proved a similar result in \cite{Uhlenbeck}. Lemma \ref{la:Uhlenbeck} follows by compactness from the following
\begin{lemma}
\label{la:uhlenbeckw22}
There are constants $\eps[m] > 0$ and $\constant[m] > 0$ such that the following holds: Let $\Omega \in W^{1,2}(D^2,so_m \otimes \R^2)$ be given with\begin{equation}
\label{eq:u2:o}
 \Vert \Omega \Vert_{L^2(D^2)} \leq \eps[m].
\end{equation}
Then there exist some $\xi \in W^{2,2}(D^2,so_m)$ with $\int_{D^2} \xi = 0$ and some $P \in W^{2,2}(D^2,SO_m)$ with $P-I \in W^{1,2}_0(D^2,\R^{m\times m})$, where $I$ denotes the identity matrix, such that 
\begin{equation}
\label{eq:u2:dec}
 \nabla^\bot \xi = P^{-1} \nabla P + P^{-1} \Omega P \mbox{\quad pointwise a.e.~in $\disk$.}
\end{equation}
In addition, we have the estimates
\begin{equation}
\label{eq:u2:estw12}
 \Vert \xi \Vert_{W^{1,2}(D^2,\R^{m\times m})} + \Vert P - I \Vert_{W^{1,2}(D^2,\R^{m\times m})} \leq \constant[m]\ \Vert \Omega \Vert_{L^2(D^2, \R^{m\times m \times 2})}
\end{equation}
and
\begin{equation}
\label{eq:u2:estw22}
 \Vert \xi \Vert_{W^{2,2}(D^2,\R^{m\times m})} + \Vert P - I \Vert_{W^{2,2}(D^2,\R^{m\times m})} \leq \constant[m]\ \Vert \Omega \Vert_{W^{1,2}(D^2,\R^{m\times m \times 2})}.
\end{equation} 
\end{lemma}
In order to proof this lemma, we introduce for yet to be chosen $\eps[m]$ and $\constant[m]$ the set
\[
 \mathcal {U} \equiv \mathcal{U}_{\eps[m],\constant[m]} = \left \{t \in [0,1]\ \left \vert\ \begin{array}{c}
\mbox{There is a decomposition of $t\Omega$ and}\\
\mbox{\eqref{eq:u2:dec}--\eqref{eq:u2:estw22} hold.}                                                                        
                                                                      \end{array}
\right. \right \}
\]
This set is clearly non-empty as $0 \in \mathcal {U}$ (using $\xi \equiv 0$ and $P \equiv I$). Furthermore it is closed, due to \eqref{eq:u2:estw22}. To prove openness we fix some $t_0 \in \mathcal {U}$, $t_0 < 1$. By definition of $\mathcal {U}$ we then find some $\zeta \equiv \xi_{t_0} \in W^{2,2}(D^2,so_m\otimes \R^2)$ and $R \equiv P_{t_0} \in W^{2,2}(D^2,SO_m)$ such that \eqref{eq:u2:dec}, \eqref{eq:u2:estw12}, \eqref{eq:u2:estw22} hold where $\xi$ and $P$ are replaced by $\zeta$ and $R$, respectively. We now prove the following
\begin{proposition}
Define the operator
\[
 T: W^{2,2}\cap W^{1,2}_0(D^2,so_m) \times W^{1,2}(D^2,so_m\otimes\R^2) \to L^2(D^2,so_m),
\]
\[
 T(U,\lambda) := div (e^{-U} \nabla e^U + e^{-U} (\nabla^\bot \zeta + \lambda) e^{U}).
\]
Then, there is a constant $\alpha > 0$ such that the following holds: If $\Vert \nabla \zeta \Vert_{L^2(D^2)} \leq \alpha$ is true, then there exists some $\gamma > 0$ such that for every $\lambda \in W^{1,2}(D^2,so_m\otimes \R^2)$ with $\Vert \lambda \Vert_{W^{1,2}(D^2,\R^{m\times m})} \leq \gamma$ we find some $U_\lambda \in W^{2,2} \cap W^{1,2}_0 (D^2,so_m)$ such that
\[
 T(U_\lambda,\lambda) = 0.
\]
Furthermore, $U_\lambda$ depends continuously on $\lambda$.
\end{proposition}
\begin{proof}
First of all, we notice that $T$ is well defined and smooth, as the exponential function maps $W^{2,2}$ into $W^{2,2}$ smoothly. Furthermore, we have $T(0,0) = 0$. The proposition follows from the implicit function theorem, if we can prove that the linearization in the first component of $T$ at $(U,\lambda)=(0,0)$, namely
\[
 H(\psi) := \lap \psi + \nabla \psi \cdot \nabla^\bot \zeta - \nabla^\bot \psi \cdot \nabla \zeta,
\]
is an isomorphism
\[
 H : W^{2,2}\cap W^{1,2}_0(D^2,so_m) \to L^2(D^2,so_m).
\]
The injectivity follows for small $\delta > 0$ as in \cite{Uhlenbeck}: For $1 < p < 2$ we have
\[
 \Vert H(\psi) \Vert_{L^p} \geq \Vert \lap \psi \Vert_{L^p} - C\Vert \nabla \psi \Vert_{L^{p^\ast}}\ \Vert \nabla \zeta \Vert_{L^2} \geq c_0 \Vert \psi \Vert_{W^{2,p}} - \Vert \psi \Vert_{W^{2,p}} \alpha,
\]
where $p^\ast = \frac{2p}{2-p}$ is the Sobolev-exponent. (By Wente's inequality, Theorem \ref{th:wente}, this follows also for $p = 2$.)

Concerning the proof of surjectivity, we note that the operator $K: W^{2,2} \cap W^{1,2}_0(D^2,\R^{m\times m}) \to W^{2,2} \cap W^{1,2}_0(D^2,\R^{m\times m})$, defined by
\[
 \begin{cases}
 \Delta K(\psi) = \nabla \psi \cdot \nabla^\bot \zeta\quad &\mbox{in $D^2$,}\\
K(\psi) = 0 &\mbox{on $\partial D^2$},
 \end{cases}
\]
is linear, bounded and compact. The compactness is seen by approximating $\zeta$ with smooth functions. Since $K$ is injective we conclude by the Fredholm alternative that $id-K$ is an isomorphism. From this we get that $H$ is surjective.\qed
\end{proof}

We proceed with the proof of Lemma \ref{la:uhlenbeckw22}: For small $\lambda \in W^{1,2}(D^2,so_m\otimes \R^2)$ we define 
\[
 Q \equiv Q_\lambda := e^{U_\lambda} \in W^{2,2}(D^2,SO_m)
\]
and notice $Q-I \in W^{1,2}_0(D^2,\R^{m\times m})$. Let $P := RQ$. By the decomposition of $t_0 \Omega$ we then obtain
\[
 div (P^{-1} \nabla P - P^{-1} (t_0\Omega + R \lambda R^{-1}) P) = 0 \mbox{\quad in $D^2$.}
\]
Setting $\lambda := R^{-1} (t - t_0) \Omega R$, which is small in $W^{1,2}$ whenever $\abs{t-t_0}$ is small, we have
\[
 div (P^{-1} \nabla P - P^{-1} t\Omega P) = 0 \mbox{\quad in $D^2$.}
\]
Therefore, the Poincar\'{e} Lemma for differential forms yields a mapping $\xi \in W^{2,2}(D^2,\R^{m\times m})$ such that
\begin{equation}\label{eq:decomp:t}
 \nabla^\bot \xi = P^{-1} \nabla P - P^{-1} t\Omega P \mbox{\quad in $D^2$}
\end{equation}
is satisfied. In addition, we can assume $\xi$ to have zero mean value on $D^2$. Writing $P=P_t$ and $\xi=\xi_t$ for the just constructed solution of (\ref{eq:decomp:t}), we note that $\Vert P_{t} - R \Vert_{W^{2,2}}$ and $\Vert \xi_{t} - \zeta\Vert_{W^{2,2}}$ are small whenever $\abs{t-t_0}$ is small. Applying (\ref{eq:u2:estw12}) for $R$ and $\zeta$ we thus conclude for any $\delta>0$: For arbitrary $t\in[0,1]$ with sufficiently small $\abs{t-t_0}$ we may choose $\eps[m]>0$ small enough in dependence of $\constant[m]$ and $\delta$ to ensure
\[
 \Vert P-I \Vert_{W^{1,2}} + \Vert \xi \Vert_{W^{1,2}} \leq \delta
\]
for the solution $P=P_t$, $\xi=\xi_t$ of (\ref{eq:decomp:t}). The estimates \eqref{eq:u2:estw12} and \eqref{eq:u2:estw22} and hence the openness of $\mathcal U$ follow then from the subsequent Proposition \ref{pr:decest} and the lemma is proven. \qed

\begin{proposition}
\label{pr:decest}
There are constants $\delta(m) \in (0,1)$ and $\constant[m] > 0$ such that the following holds: Let $P \in W^{2,2}(D^2,SO_m)$, $P-I \in W^{1,2}_0(D^2,\R^{m\times m})$, $\xi \in W^{2,2}(D^2,\R^{m\times m})$, $\mvint_{D^2} \xi = 0$ and $\Omega \in W^{1,2}(D^2,\R^{m\times m})$ satisfy
\begin{equation}
\label{eq:decest:dec}
 \nabla^\bot \xi = P^{-1} \nabla P + P^{-1} \Omega P \mbox{\quad in $D^2$.}
\end{equation}
If moreover the estimate
\begin{equation}
\label{eq:decest:klein}
 \Vert \Omega \Vert_{L^2(D^2)} + \Vert P-I \Vert_{W^{1,2}(D^2)} + \Vert \xi \Vert_{W^{1,2}(D^2)} \leq \delta(m)
\end{equation}
is satisfied, then \eqref{eq:u2:estw12} and \eqref{eq:u2:estw22} hold as well.
\end{proposition}
\begin{proof}
Multiplying \eqref{eq:decest:dec} by $P$ yields
\begin{equation}
\label{eq:decest:pdep}
 \lap (P-I) = \nabla (P-I) \cdot \nabla^\bot \xi - div (\Omega\ P)  \mbox{\quad in $D^2$.}
\end{equation}
Wente's inequality, Theorem \ref{th:wente}, and $L^2$-theory implies
\[
 \Vert \nabla (P-I) \Vert_{L^2} \leq C \big(\Vert \nabla (P-I) \Vert_{L^2}\ \Vert \nabla \xi \Vert_{L^2} + \Vert \Omega \Vert_{L^2}\big).
\]
Here we used crucially that $\Vert P^{-1} \Vert_{L^\infty} + \Vert P \Vert_{L^\infty}$ is a-priori bounded by a constant, since $P \in SO_m$ holds pointwise a.e.~in $D^2$. The estimate \eqref{eq:u2:estw12} then follows by taking the $L^2$-norm in \eqref{eq:decest:dec} and choosing $\delta(m)$ sufficiently small.

For the proof of \eqref{eq:u2:estw22} we start with the obvious inequality
\[
 \Vert div(\Omega P) \Vert_{L^2} \leq C (\Vert \Omega \Vert_{W^{1,2}}+ \Vert \Omega\cdot \nabla P \Vert_{L^2}).
\]
Using the imbedding $W^{1,\frac{n}{2}} \hookrightarrow L^n$ we get in two dimensions:
\[
 \Vert \Omega\cdot \nabla P \Vert_{L^2} \leq C \ \Vert \Omega\cdot \nabla P \Vert_{W^{1,1}} \leq C \  (\Vert \Omega \Vert_{W^{1,2}}\ \Vert \nabla P \Vert_{L^2} + \Vert \Omega \Vert_{L^2}\ \Vert \nabla^2 P \Vert_{L^2}).
\]
Furthermore, the above mentioned imbedding implies
\[
 \Vert \nabla P \cdot \nabla^\bot \xi \Vert_{L^2} \leq C \Vert \nabla P\Vert_{L^2}\ \Vert \nabla \xi \Vert_{L^2} + \Vert \nabla^2 P \Vert_{L^2}\ \Vert \nabla \xi \Vert_{L^2} + \Vert \nabla P \Vert_{L^2}\ \Vert \nabla^2 \xi \Vert_{L^2}.
\]
Using \eqref{eq:decest:klein}, \eqref{eq:u2:estw12} and $\delta=\delta(m) < 1$, we infer
\[
 \Vert \nabla P\cdot \nabla^\bot \xi \Vert_{L^2} \leq C \big(\Vert \Omega \Vert_{W^{1,2}} + \delta \Vert \nabla^2 P \Vert_{L^2} + \delta \Vert \nabla^2 \xi \Vert_{L^2}\big).
\]
Thus by $L^2$-theory and \eqref{eq:decest:pdep} we have
\begin{equation}
\label{eq:decest:pw22}
 \Vert P-I\Vert_{W^{2,2}} \leq C \delta (\Vert \nabla^2 P \Vert_{L^2} + \Vert \nabla^2 \xi \Vert_{L^2}) + C \Vert \Omega \Vert_{W^{1,2}}.
\end{equation}
Starting from \eqref{eq:decest:dec}, we obtain by the same techniques:
\[
 \Vert \nabla \xi\Vert_{W^{1,2}} \leq C (\Vert \Omega \Vert_{W^{1,2}} + \Vert \nabla^2 P \Vert_{L^2}).
\]
Choosing $\delta > 0$ small enough and employing \eqref{eq:decest:pw22}, we finally arrive at \eqref{eq:u2:estw22}. \qed
\end{proof}

\begin{remark}
A similar result holds as well for $\Omega \in W^{2,n}(D^n,so_m \otimes \R^n)$.
\end{remark}

\subsection{Dirichlet growth theorem}\label{dirichlet}
In this section we sketch the \emph{proof of Proposition \ref{pr:dirichlet}}. It is very similar to the one of Theorem 3.5.2 in Morrey's monograph \cite{Morrey}.

Recall the presupposed relation 
\[\tag{\ref{eq:bound:dirgrowth}}
 \mathcal{M}_p(x_0,r;u) \leq \constant[s,p,\gamma] \Big[\mathcal{M}_p(x_0,R;u) + \Vert e \Vert_{L^s(B_R(x_0))}^p\Big] \left (\frac{r}{R} \right )^{\mu}
\]
for all $x_0\in D^2$ and $0 < r < R < R_0$ with $B_R(x_0) \subset D^2$. Moreover, H\"older's inequality implies
\begin{equation}\label{eq:dgt:Mpu}
  \mathcal{M}_p(x_0,R;u) \leq C \Vert  \nabla u \Vert_{L^2(B_{R}(x_0))}^p.
\end{equation}
We fix some $a\in D^2$ and $\varrho \in (0,\frac{R_0}2)$ with $B_{2\varrho}(a) \subset D^2$ and pick arbitrary $x, y \in B_{\varrho}(a)$. For any $\eta \in B_{\varrho}(a)$ we have
\[
 \abs{u(x)- u(y)} \leq \abs{u(x)-u(\eta)} + \abs{u(y) - u(\eta)},
\]
and hence
\[
 \abs{u(x)- u(y)} = \mvint_{B_{\varrho}(a)}\!\abs{u(x)- u(y)}\,d\eta \leq \mvint_{B_{\varrho}(a)}\!\big(\abs{u(x)-u(\eta)} + \abs{u(y) - u(\eta)}\big)\,d\eta.
\]
Defining the point $x_t := x + t(a-x)$ for $t \in (0,1)$,  we now calculate
\begin{eqnarray*}
&&  \hspace*{-9ex}\frac{1}{\varrho^2}\int_{B_{\varrho}(a)} \abs{u(x)-u(\eta)}\,d\eta\\
&\qquad &\leq \ \int_{B_{\varrho}(a)} \left(\int_{0}^1 \frac{1}{\varrho} \abs{\nabla u(x+t(\eta - x))}\ dt\right)\,d\eta\\
&& \leq \ C \frac{1}{\varrho}\ \int_0^1 \left (\int_{B_{\varrho}(a)} \abs{\nabla u(x+t(\eta - x))}^p \,d\eta\right )^{\frac{1}{p}}\ \varrho^{2(1-\frac{1}{p})}\,dt\\
&& \leq \ C \int_0^1\varrho^{1 - \frac{2}{p}}\ t^{-\frac{2}{p}} \left (\int_{B_{t\varrho} (x+t(a-x))} \abs{\nabla u}^p \right )^{\frac{1}{p}}\,dt\\ 
&& = \ C \int_{0}^1\varrho^{1 - \frac{2}{p}}\ t^{-\frac{2}{p}} \left ((t\varrho)^{2-p}\ J_p(x_t,t\varrho; u) \right )^{\frac{1}{p}}\,dt\\
&& = \ C \int_{0}^1 t^{-1} \left (J_p(x_t,t\varrho; u) \right )^{\frac{1}{p}}\,dt\\
&& \hspace*{-1.5ex}\overset{\eqref{eq:bound:dirgrowth}\atop \eqref{eq:dgt:Mpu}}{\leq} C \int_{0}^1 t^{-1} \left ( \left (\frac{t\varrho}{\varrho} \right )^{\mu} (C \Vert \nabla u \Vert_{L^2(B_{\varrho}(x_t))}^p + \Vert e \Vert_{L^s(B_\varrho(x_t))}^p \right )^{\frac{1}{p}}\,dt\\[1ex]
&&\leq \ C \int_{0}^1 t^{-1 + \frac{\mu}{p}}\ \left ( \Vert \nabla u \Vert_{L^2(B_{2\varrho}(a))} + \Vert e \Vert_{L^s(B_{2\varrho}(a)}) \right )\,dt\\[1ex]
&& =  \ C \frac p\mu \left ( \Vert \nabla u \Vert_{L^2(B_{2\varrho}(a))} + \Vert e \Vert_{L^s(B_{2\varrho}(a))}  \right ). 
\end{eqnarray*}
This gives
\[
 \abs{u(x)-u(y)} \leq C \frac p\mu \left (\Vert \nabla u \Vert_{L^2(B_{2\varrho}(a))} + \Vert e \Vert_{L^s(B_{2\varrho}(a) )} \right )^{\frac{1}{p}} \mbox{\quad for all $x,y \in B_{\varrho}(a)$}, 
\]
and the proposition follows by replacing $\varrho$ by $\frac\varrho2$.\qed

\subsection{Continuity on the boundary}\label{strzcont}
We conclude with recalling the \emph{proof of Lemma \ref{la:bound:strzlecki}}, which widely agrees with the proof of Lemma 3.1 in \cite{Strzelecki03}; see also \cite{HildKaul72}, Lemma 3.

For $\varrho \in [0,1]$, $\theta \in [0,2\pi)$ let 
\[
v(\varrho,\theta) := u(\varrho\cos \theta,\varrho\sin \theta).
\]
We denote the continuous representation of the trace $u \big \vert_{\partial D^2}$ with $\psi$. Let us fix $y_0 = (\cos\theta_0,\sin \theta_0) \in \partial D^2$ and let $x_1 = \varrho_1 e^{i\theta_1}$ be an interior point of $D^2$. We pick some $x\strich =\varrho_1 e^{i\theta\strich}\in B_{\frac{\delta}{2}}(x_1)$, where $\theta\strich$ will be chosen later and $\delta := 1 - \varrho_1$. Setting
	$$y\strich:=\frac{x\strich}{\abs{x\strich}} = (\cos \theta\strich,\sin \theta \strich) \in \partial D^2,$$ 
we then have
\[
 \abs{u(x_1) - \psi(y_0)} \leq \abs{u(x_1) - u(x\strich)} + \abs{u(x\strich) - \psi(y\strich)} + \abs{\psi(y\strich) - \psi(y_0)}.
\]
For small $\delta$ and small $\abs{\theta_0 - \theta_1}$ the third term becomes small. Assumption \eqref{eq:strz:osc} implies that, for small $\delta=1-\abs{x_1}$, the first term becomes small as well. Hence, we only have to check the smallness of
\[
 \abs{u(x\strich) - \psi(y\strich)}
\]
for small $\delta$, where $x\strich =\varrho_1 e^{i\theta\strich}\in B_{\frac{\delta}{2}}(x_1)$ is yet to be chosen. This can be done exactly as in \cite{Strzelecki03}: For any $\sigma \in (0,2\pi)$ we have
\[
 \int_{\theta}^{\theta+\sigma} \int_{1-\delta}^1 \abs{v_r(r,\vartheta)}^2\ r\ dr\ d\vartheta \leq \int_{1-\delta \leq \abs{x} \leq 1} \abs{\nabla u}^2 =: I(\delta),
\]
which implies
\[
 \int_{\theta}^{\theta + \sigma} \int_{1-\delta}^1 \abs{v_r (r,\vartheta)}^2\ \frac{r}{r}\ dr\ d\vartheta \leq \frac{I(\delta)}{1-\delta}.
\]
By a contradiction argument we obtain the existence of a set $E_\sigma \subset (\theta,\theta + \sigma)$ with positive one-dimensional Lebesgue measure and such that
\begin{equation}
\label{eq:strz:3.2}
 \int_{1-\delta}^1 \abs{v_r(r,\tilde{\theta})}^2 \,dr\leq \frac{I(\delta)}{\sigma(1-\delta)} \quad \mbox{for all $\tilde{\theta} \in E_\sigma$}
\end{equation}
is true. By approximation we can assume that, additionally,
\[
 \abs{\psi(\cos \tilde{\theta}, \sin \tilde{\theta}) - v(\varrho_1,\tilde{\theta})} = \abs{v(1,\tilde{\theta}) - v(\varrho_1,\tilde{\theta})}  \leq \int_{\varrho_1}^1 \abs{v_r(r,\tilde{\theta})}\ dr
\]
holds for all $\tilde{\theta} \in E_\sigma$. Hence, we can estimate
\[
 \abs{\psi(\cos \tilde{\theta}, \sin \tilde{\theta}) - v(\varrho_1,\tilde{\theta})} \leq (1-\varrho_1)^{\frac{1}{2}}\! \left ( \int_{\varrho_1}^1 \abs{v_r(r,\tilde{\theta})}^2dr\!\right )^{\frac{1}{2}} \overset{\eqref{eq:strz:3.2}}{\leq} \left (\frac{\delta}{\sigma} \right )^{\frac{1}{2}}\!\frac{\sqrt{I(\delta)}}{(1-\delta)^{\frac{1}{2}}}.
\]
By setting $\sigma := \frac{\delta}{4}$ and choosing $\theta\strich\in E_{\frac{\delta}{4}}$ we arrive at
\[
 \abs{u(x\strich) - u(y\strich)} \leq 2 \left (\frac{I(\delta)}{1-\delta}\right )^{\frac{1}{2}} \xrightarrow{\delta \to 0} 0,
\]
and Lemma \ref{la:bound:strzlecki} is proven. \qed